\documentclass[11pt,a4paper]{article}

\usepackage[total={7in, 8in}]{geometry}
\usepackage[utf8]{inputenc}
\usepackage{graphicx}%
\usepackage{amsmath,amssymb,amsfonts}%
\usepackage{amsthm}
\usepackage{mathrsfs}%
\usepackage[title]{appendix}%
\usepackage{xcolor}%
\usepackage{textcomp}%
\usepackage{manyfoot}%
\usepackage{booktabs}%
\usepackage{algorithm}%
\usepackage{algorithmicx}%
\usepackage{algpseudocode}%
\usepackage{listings}%
\usepackage{physics}
\usepackage{dsfont}
\usepackage[hidelinks]{hyperref}
\usepackage{subcaption}
\usepackage[onehalfspacing]{setspace}
\usepackage[backend=biber, style=numeric, sorting=nty]{biblatex}

\newcommand{\footremember}[2]{%
    \footnote{#2}
    \newcounter{#1}
    \setcounter{#1}{\value{footnote}}%
}
\newcommand{\footrecall}[1]{%
    \footnotemark[\value{#1}]%
} 

\newtheorem{theorem}{Theorem}[section]
\newtheorem{lemma}[theorem]{Lemma}
\newtheorem{prop}[theorem]{Proposition}
\newtheorem{cor}[theorem]{Corollary}

\theoremstyle{remark}

\newtheorem*{remark}{Remark}
\newtheorem*{notation}{Notation}

\theoremstyle{definition}
\newtheorem{definition}[theorem]{Definition}

\DeclareMathOperator{\Cov}{Cov}
\DeclareMathOperator*{\argmin}{argmin}

\DeclareMathOperator{\dist}{dist}
\DeclareMathOperator{\Vol}{Vol}
\DeclareMathOperator{\dVol}{dVol}
\DeclareMathOperator{\id}{id}
\DeclareMathOperator{\St}{St}
\DeclareMathOperator{\Gr}{Gr}
\DeclareMathOperator{\inv}{inv}
\DeclareMathOperator{\diam}{diam}
\newcommand{\ddf}[3]{\frac{\dd^{#1} #2}{\dd #3^{#1}}}
\newcommand{\eqdist}{\overset{\mathrm{d}}{=}}

\title{Autoregressive Processes on Stiefel and Grassmann Manifolds}

\author{Jordi-Llu\'{i}s Figueras\footremember{upps}{Department of Mathematics, Uppsala University}\footremember{cat}{Department of Mathematics, Polytechnic University of Catalonia} \\
\small{\texttt{jordi.lluis.figueras@upc.edu}}
\and
Aron Persson \footrecall{upps} \footremember{umu}{Department of Mathematics and Mathematical Statistics, Ume{\aa} University} \footremember{huawei}{Partially funded by Huawei technologies.} \footnote{Corresponding Author}   \\
\small{\texttt{aron.persson@umu.se}}
}
\date{}

\begin{document}

\maketitle

\abstract{System identification methods for autoregressive processes on Stiefel and Grassmann manifolds are presented and studied. We model the system parameters as elements of the orthogonal group and we show that the system can be estimated by averaging over observations. Lastly, we propose an algorithm for computing these system parameters using a conjugate gradient method adapted to Stiefel and Grassmann manifolds.}



\section{Introduction}
\label{sec:intro}

In modern signal processing and MIMO systems, a state of $n$ antennas and $k$ receivers is described by a general $n$ by $k$ matrix. But it is only the orthogonal part of the state that is measured, i.e. one measures the orthogonal part of the polar decomposition of a $n$ by $k$ matrix. In other words the measurements are points on Grassmann or Stiefel manifolds, depending on application and context, see for example \cite{Kim2020stochcons,kong2022momentum,li2013riemannian,schwarz2014predictive,solo2014approach,xue2020blind}.

Now consider a sequence of measurements over time, a time series that takes values on a Grassmann or Stiefel manifold. If one now wants to make a prediction of future states or analyze the dynamics of such a time series, we have to make a suitable model for the time series. In $\mathbb{R}$-valued statistics, the most basic models used in time series analysis are the autoregressive processes. This raises the question: What is a natural definition of an autoregressive process on the Stiefel and Grassmann manifolds? Moreover, how can one estimate the parameters of the model in a reasonable manner? In particular, how does one predict a future state given that the $k$ receivers move around?

In \cite{xavier2006generalization}, Xavier and Manton introduced some AR-processes on Riemannian manifolds, a treatise we draw inspiration from and aim to specialize to the class of homogeneous spaces. The goal is to enable us to estimate the Lie group-valued system parameters on Grassmann and Stiefel manifolds. Some related approaches were considered in \cite{schwarz2012adaptive}, following ideas similar to those of Xavier and Manton. However, by definition, a homogeneous space has a transitive Lie-group action, this naturally guides us to consider system parameters as elements in the Lie group. This is done in contrast to \cite{schwarz2012adaptive,xavier2006generalization} where the autoregressive system parameters are linear transformations on the tangent spaces. In Proposition \ref{prop:Argroup}, we show that one may, somewhat equivalently to previous models, model AR($K$)-processes on Stiefel and Grassmann manifolds by considering the system parameter as an element in the orthogonal group, see Equation \eqref{eq:autoregressivehomo}. Moreover, by Theorem \ref{thm:ARexpect} we show that the system parameter can be estimated, and numerically computed, by averaging over the observations. Using the definition for the expectation on manifolds, we propose a conjugate gradient descent method in Algorithm \ref{alg:stiefelgrad}, which utilizes an approximation of the distance function for the Grassmann and Stiefel manifolds.

This paper is relevant to the wide field of directional statistics, the most common form is that of directional statistics on the spheres $\mathbb{S}^n$. The spheres are special cases of the Stiefel manifolds, more precisely $\mathbb{S}^n \cong \St_{n+1,1}$. Ever since Fischer in \cite{fisher1953dispersion} popularized the idea of modeling data on spheres in order to ascertain geological information on rocks, see also \cite{mardia2025fisher}, there has been an abundance of applications of directional statistics. Directional statistics has found its way into meteorology \cite{bowers2000directional,li2018wavefront,nunez2015bayesian,Nunez-Gutierrez-Escarela-2011}, into medicine \cite{demir2019application,karaibrahimoglu2021circular} and into machine learning \cite{kanamori2012non}, to name a few.

We start, in Section \ref{sec:AutRM}, by giving an overview on what has been done previously in the area of autoregressive processes on Riemannian manifolds, and general statistics on manifolds. Then, in Section \ref{sec:ARhomo}, we specialize to homogeneous spaces and we give some new results. Thereafter, in Section \ref{sec:conjgrad}, we give some algorithms that estimate and approximate the system parameters on the Stiefel and Grassmann manifold using the conjugate gradient method. A fundamental introduction to Riemannian geometry is given in Appendix \ref{sec:basicriemgeom} that introduces concepts needed in order to understand the reasoning in this paper. In the second part of the appendix, Appendix \ref{sec:simulations}, one can find some figures showing simulations on some Stiefel and Grassmann manifolds of various dimensions that illustrate Algorithm \ref{alg:stiefelgrad} given in Section \ref{sec:conjgrad}.

 \section*{Notation}

\begin{itemize}
    \item $\mathbb{P}$, probability measure.
    \item $(\Omega,\mathcal{A},\mathbb{P})$, probability space.
    \item $\mathbb{E}$, the Fr\'echet mean.
    \item $\exp_p$, the geodesic exponential map at $p$.
    \item $\log_p$, the geodesic logarithmic map at $p$, defined on the maximal star-shaped domain of injectivity.
    \item $\exp_{\mathrm{M}}$, the matrix exponential.
    \item $O(n)$, the orthogonal group of $n\times n$ matrices.
    \item $\St_{n,k}$, the Stiefel manifold.
    \item $\dist$, the Riemannian distance function. 
    \item $T_pM$, the tangent space of $M$ at $p$.
    \item $P_{p,q}$, parallel transport along geodesic from $p$ to $q$.
    \item $g_p$, Riemannian metric.
    \item $\Vol$, volume form and measure induced by Riemannian metric.
    \item $\norm{v}$, the norm induced by the Riemannian metric of a tangent vector $v$.
    \item $I_n$, identity $n\times n$ matrix.
    \item $\mathds{1}_A$, indicator function (i.e. characteristic function) of a set $A$, $\mathds{1}_A(x)=1$ if $x\in A$ and $\mathds{1}_A(x)=0$ if $x\notin A$.
    \item $B_{r}(x)$, geodesic ball of radius $r$ and center $x$.
    \item $\Phi x$, with $\Phi\in G$ a Lie group and $x\in G/H$ a homogeneous space, denotes the Lie group action.
    \item $\eqdist$, equal in distribution.
\end{itemize}

\section{Statistics on Riemannian Manifolds}
\label{sec:AutRM}

Throughout, let $(\Omega, \mathcal{A},\mathbb{P})$ be a probability space. Let $M$ be a Riemannian manifold. We shall recall a slightly modified, but still equivalent, version of the "wrapped process" definition of an AR-process found in \cite{xavier2006generalization}. Denote $\exp_p$ and $\log_p$ to be the Riemannian exponential map centered at $p$ and its inverse, respectively, see Definition \ref{def:geodexplog}. Moreover, the parallel transport from $p\in M$ to $q\in M$ along a geodesic from $p$ to $q$ is written as $P_{p,q}:T_pM\rightarrow T_q M$. All expressions involving logarithm maps are understood to hold almost surely, so that any random point lies inside the image of the injectivity domain of the exponential map with probability 1.

\begin{definition}
\label{def:arprocess}
Let $M$ be a connected and geodesically complete Riemannian manifold, and let $Z_\ell\in M$, $\ell\in \mathbb{Z}$ be a time series. Let $K \in \mathbb{N}$, and suppose there are $K$ (1,1)-tensor fields $A_j:TM\rightarrow TM$, i.e.  a smooth bundle map such that for every $p\in M$ it holds that
\[
A_j\mid_{T_p M}:T_pM \rightarrow T_pM
\]
is linear. Then, an \emph{AR($K$)-process} on $M$ is defined as
\[
Z_\ell = \exp_{Z_{\ell-1}}\Big(\varepsilon_\ell+ \sum_{j=1}^K A_j P_{Z_{\ell-j},Z_{\ell-1}}( \log_{Z_{\ell-j}}(Z_{\ell-j-1})) \Big),    
\]
where $\varepsilon_\ell:\Omega \rightarrow T_{Z_{\ell-1}}M$ is a random variable with $\mathbb{E}[\varepsilon_\ell]=0$, $\Cov(\varepsilon_\ell)=\id_{T_{Z_{\ell-1}}M}$.
\end{definition}

\begin{remark}
    Note that in the case that $M= \mathbb{R}^n$ is the Euclidean space an AR($K$) process by this definition is
    \[
    Z_\ell =Z_{\ell-1} + \sum_{j=1}^K A_j (Z_{\ell-j-1}-Z_{\ell-j}) + \varepsilon_\ell
    \]
    which is not directly analogous to the classical AR($K$)-process as here $Z_\ell$ depends on $Z_{\ell-1}, \dots, Z_{\ell-N-1}$ whilst the classical AR($K$)-process only depends on $Z_{\ell-1},\dots, Z_{\ell-N}$. However this way of modeling is necessary if one wants to lift the model parameters as linear operators on the tangent space. See \cite{xavier2006generalization} for a more in depth discussion about this model.
\end{remark}

For our purposes we need the notion of expected value of a Riemann manifold-valued random variable. This is precisely the well known least square minimizer on $\mathbb{R}^n$ adapted to Riemannian manifolds. Suppose $\dist$ is the Riemannian distance function on $M$ and $\Vol$ is the Riemannian volume measure on $M$. Following \cite{dubey2019frechet,pennec2006intrinsic}, the Fr\'echet mean of an absolutely continuous random variable $X:(\Omega,\mathcal{F},\mathbb{P}) \rightarrow M$ is then any point $y\in M$ satisfying
\begin{equation}
\mathbb{E}[X] = \argmin_{y\in M} \int_M \dist^2(y,x)p_X(x) \dVol(x).
\label{eq:expectation}
\end{equation}

\begin{notation}
    Throughout, the notation $\mathbb{E}$ always denotes the Fr\'echet mean. Whenever the expectation is taken of an absolutely continuous random variable on Euclidean space, the Fr\'echet mean coincides with the classical probability-theoretical expectation. Since we only consider absolutely continuous random variables in this paper there is no conflict of notation.
\end{notation}

Here $p_X = \frac{\dd X_*\mathbb{P}}{\dVol}$ is the Radon-Nikodyn derivative of the push forward measure, defined by $X_*\mathbb{P}(A)= \mathbb{P}(X^{-1}(A))$ for any Borel subset $A\subseteq M$, with respect to the Riemannian volume measure. Note that in contrast to random variables on $\mathbb{R}^n$, the expected value need not be uniquely defined. This is clearly illustrated by considering a uniformly distributed random variable on $\mathbb{S}^2$. In this case the probability density function is constant, then any, and thus all, points on $\mathbb{S}^2$ minimize the integral above. { In many cases the Fr\'echet mean is not tractable in such cases the exponential barycenter is also commonly used \cite{bouchard2025beyond,pennec2013exponential}. A point $\mu \in M$ is an \textit{exponential barycenter} for an absolutely continuous random variable $X:\Omega \rightarrow M$ if
\[
0 = \int_M \log_\mu(x) p_X(x) \dVol(x).
\]
Note that a Fr\'echet mean is always an exponential barycenter, given sufficient regularity of the distribution, but the reverse is not always true (the reverse holds on Euclidean space).
}

\section{AR-Processes on Homogeneous Spaces by Lie Group Action}
\label{sec:ARhomo}

Let $G$ be a compact Lie group with bi-invariant Riemannian metric, and let $M$ be a geodesically complete $G$-homogeneous space with metric induced by $G$. See Appendix \ref{sec:basicriemgeom} for the definition. Throughout, we denote the Lie group action $\Phi$ unto a $G$-homogeneous space by $\Phi x$ whenever $x\in M$.

\begin{prop}
\label{prop:Argroup}
Suppose $Z_\ell$ is an AR($K$)-process on $M$ as in Definition \ref{def:arprocess} with system parameters $A_j$.
Then, for every $1\leq j\leq K$ and any pair of fixed points $Z_{\ell-j},Z_{\ell-j-1}\in M$ there is a $\Phi_j\in G$ such that
\begin{equation}
\Phi_j Z_{\ell-j} = \exp_{Z_{\ell-j}}(A_j \log_{Z_{\ell-j}}(Z_{\ell-j-1})).
\label{eq:liegrouptensorfield}
\end{equation}
\end{prop}

\begin{proof}
Denote
\[
q = \exp_{Z_{\ell-j}}(A_j \log_{Z_{\ell-j}}(Z_{\ell-j-1})).
\]
Since the Lie group action on homogeneous spaces is transitive there exists an element $\Phi_j \in  G$ such that $q= \Phi_j Z_{\ell-j}$.
\end{proof}

\begin{remark}
One important reason one may want to consider system parameters as Lie group elements is that for Stiefel and Grassmann manifolds the degrees of freedom for finding parameters $\Phi \in O(n)$ is much smaller than the degrees of freedom for finding tensor fields $A_j\in \mathcal{T}^1_1(M)$ so they will in general be less taxing to compute. Moreover, what the proof above indicates is that due to a high degree of multicollinearity (i.e. a large amount of the components are correlated) of the system parameters, we are motivated to model the system parameters as elements in the Lie group instead.
\end{remark}

Therefore we are motivated to consider the following definition:

\begin{definition}
Let $M$ be a $m$-dimensional Riemannian $G$-homogeneous space. \emph{An autoregressive process of order $K$},  denoted AR($K$), is a process $Z_\ell$ satisfying 
\begin{equation}
\begin{aligned}
    \kappa_\ell &= \exp_{Z_{\ell-1}}\Big(\sum_{j=1}^K P_{Z_{\ell-j},Z_{\ell-1}}\log_{Z_{\ell-j}}(\Phi_j Z_{\ell-j}) \Big)\\
    Z_\ell&= \exp_{\kappa_\ell} \Big(\mathcal{F}_{\ell}(\kappa_\ell, \varepsilon_\ell )\Big)
\end{aligned}
\label{eq:autoregressivehomo}    
\end{equation}
where  $\varepsilon_\ell\in \mathbb{R}^m$ are i.i.d. random variables with $ \mathbb{E}[\varepsilon_\ell]=0$, $\varepsilon_\ell$ is independent of $Z_j$ for all $j\leq \ell-1$ and where $\mathcal{F}_\ell:M\times \mathbb{R}^m \rightarrow TM$ is a smooth map such that $\mathcal{F}_\ell(x,v) \in T_xM$ for all $v\in \mathbb{R}^m$.
\end{definition}

\begin{remark}
In the case that $M=\mathbb{R}^n$ is the Euclidean space this type of AR($K$) process can be reduced to
\[
Z_\ell = Z_{\ell-1} + \sum_{j=1}^K \Phi_j Z_{\ell-j}- Z_{\ell-j} + \varepsilon_\ell
\]
where $\Phi_j$ are elements of the Euclidean group. Here $\mathcal{F}(u,\varepsilon_\ell)= \varepsilon_\ell$ by identifying $\mathbb{R}^m$ with the tangent space of $\mathbb{R}^m$ at $u$. The upside is that this type AR($K$)-processes only depend on the previous $K$ points (instead of the previous $N+1$ points as in Definition \ref{def:arprocess}). The downside is that this type of AR($K$)-processes does not allow for stationarity as the system parameters are Euclidean transformations.
\end{remark}

In what follows, we shall check some properties of AR(1)-processes that are preserved and under what circumstances the noise can be averaged out.

\begin{definition}
Let $M$ be a complete Riemannian manifold. { A} function $f:M \rightarrow \mathbb{R}$ is said to be \textit{reflection symmetric} around $x\in M$ if
\[
f(\exp_x(V))=f(\exp_x(-V))
\]
for all $V,-V\in T_x M$ within the domain of injectivity of $\exp_x$.
\end{definition}

{For a point $x\in M$ denote $\mathcal{S}_x$ as the mapping that reflects all unique geodesics originating from $x$, that is
\begin{equation}
    \mathcal{S}_x (\exp_x(V)) = \exp_x(-V)
    \label{eq:localreflection}
\end{equation}
defined whenever $V,-V\in T_x M$ lies within the domain of injectivity of $\exp_x$. Note that a probability density function $p_X$ for an absolutely continuous random variable $X:\Omega \rightarrow M$ is reflection symmetric around $x$ if and only if $p_X(\mathcal{S}_x(y))=p_X(y)$. Therefore, an absolutely continuous random variable $X:\Omega \rightarrow M$ is referred to as reflection symmetric if its probability density function is.
Note that the property of reflection symmetry is a weaker notion compared to the more commonly occurring property of radial symmetry found in the literature.
\begin{definition}
    Let $M$ be a geodesically complete Riemannian manifold and let $s:M\rightarrow \mathbb{R}$ be a function. Then, $s$ is said to be \textit{radially symmetric} around $x\in M$ if
    \[
    s(\exp_x(V))=  s(\exp_x(W))
    \]
    for all $V,W\in T_xM$ such that $\norm{V}=\norm{W}$.
\end{definition} 
}
This motivates the following definition for random variables: An absolutely continuous random variable is called radially symmetric if its probability density function is radially symmetric.

 Next is a lemma which confirms that the Lie group action of reflection-symmetric random variables are still reflection-symmetric. 

{
\begin{lemma}
\label{lem:reflsymmconservedunderiso}
    Let $M$ be a Riemannian $G$-homogeneous space with $G$ a compact matrix Lie group. Let $X:\Omega \rightarrow M$ be an  absolutely continuous random variable that is reflection symmetric around $\mu\in M$. For any $\Psi \in G$ the random variable
    \(
    \Psi X
    \)
    is reflection symmetric around $\Psi \mu$.
\end{lemma}

\begin{proof}
    Define the map $f:M\rightarrow M$ by $f(p)=\Psi p$. The probability density function for $f(X)$ is computed by the standard change of variables formula for integrals
\[
p_{f(X)}(x) = p_{X}(\Psi^{-1}(x)) \abs{\det (D(f^{-1}))(x)}.
\]
Note that by definition $p_X(\exp_\mu(V))= p_X(\exp_\mu(-V))$ for all $V\in T_\mu M$ within the domain of injectivity of the exponential map. Moreover, since $\Psi$ acts as an isometry for every such $V\in T_\mu M$ there exists a $W\in T_{\Psi \mu}$ such that $\Psi\exp_\mu (V) = \exp_{\Psi \mu}(W)$ and it must hold that $\Psi\exp_\mu (-V) = \exp_{\Psi \mu}(-W)$. Also note that since $G$ is a compact matrix Lie group the map $f:M \rightarrow M$ defined by $f(x)= \Psi x$ has Jacobian determinant such that $\abs{\det (Df)(x)}=1$. Therefore,
\begin{multline*}
p_{\Psi X}(\exp_{\Psi \mu}(W)) =p_{X}(\Psi^{-1} \exp_{\Psi \mu}(W)) = p_X(\exp_\mu (V))\\= p_X(\exp_\mu(-V)) =p_X(\Psi^{-1} \exp_{\Psi \mu }(-W)) = p_{\Psi X}(\exp_{\Psi \mu}(-W))
\end{multline*}
for all $W\in T_{\Psi \mu}M$ within the domain of injectivity of the exponential map.
\end{proof}
}

The following lemma is a well-established fact within differential geometry which we state for completeness.

\begin{lemma} 
\label{lem:symmetric-injectivity-domain} 
Let \(G\) be a Lie group with bi-invariant Riemannian metric. Let \(M=G/H\) be a Riemannian $G$-homogeneous space with Riemannian metric inherited from \(G\). For any \(x\in M\) let $U$ denote the maximal domain of injectivity of \(\exp_x\), then it holds that $U=-U$ as sets. In particular, if $v\in U$ then $-v\in U$ and $\mathcal{S}_x \exp_x(U)= \exp_x(U)$ where $\mathcal{S}_x$ is the local reflection map in \eqref{eq:localreflection}.
\end{lemma} 

\begin{proof} Consider an arbitrary base point \(x\in M\). Let \[ \mathfrak g=\mathfrak h\oplus\mathfrak m \] be a naturally reductive decomposition of $\mathfrak{g}$ from the bi-invariant metric on $G$. For \(v\in\mathfrak m\), it is an established fact \cite{dohira1995geodesics} that the geodesic with initial velocity \(v x: = \frac{\dd}{\dd t} \exp_G(tv) x \big|_{t=0}\in T_xM\) can be written in the form
\[
\gamma_v(t)=\exp_G(tv) x.
\] Suppose that \(\gamma_v|_{t}\) is minimizing for all $t\in [0,T]$. Then 
\[
\dist_M\bigl(x,\exp_G(Tv) x\bigr)=T\|v\|. 
\]
Since the Lie group element \(\exp_G(Tv)\) acts as an isometry on $M$ , it holds that 
\[
\dist\bigl(x,\exp_G(-Tv)x\bigr) = \dist\bigl(\exp_G(Tv)x,x\bigr)= T\|v\|. 
\]
Consequently, the curve \[ t\longmapsto \exp_G(-tv)x \] has length \(T\|v\|\) on the interval \([0,T]\), and is therefore minimizing as well. Hence, if $v$ is in the domain of injectivity for $\exp_x$ then so is $-v$.
\end{proof}

{
Note that in the special case of AR(1)-processes modeled as in \eqref{eq:autoregressivehomo} simplify to
\[
Z_\ell = \exp_{\kappa_\ell} (\mathcal{F}_\ell(\kappa_\ell,\varepsilon_\ell)) =\exp_{\Phi Z_{\ell-1}}(\mathcal{F}_\ell(\Phi Z_{\ell-1},\varepsilon_{\ell}))
\]
since $\kappa_\ell = \exp_{Z_{\ell-1}}(\log_{Z_{\ell-1}}(\Phi Z_{\ell-1})) = \Phi Z_{\ell-1}$ almost surely. Now we are ready for the first main theorem.

\begin{theorem}
    \label{thm:ARreflectionsymmetric}
Let $M$ be a connected $m$-dimensional $G$-homogeneous space with $G$ a compact Riemannian Lie group with bi-invariant metric and with Riemannian metric on $M$ inherited from $G$. Let
\[
Z_{\ell}= \exp_{\Phi Z_{\ell-1}}(\mathcal{F}_\ell(\Phi Z_{\ell-1},\varepsilon_\ell))
\]
be an AR(1)-process on $M$ as in \eqref{eq:autoregressivehomo}. Suppose that $\varepsilon_\ell$ is absolutely continuous on $\mathbb{R}^m$ and assume both $Z_{\ell-1},Z_\ell$ are absolutely continuous on $M$. Suppose the identity
\begin{equation}
\exp_{\mathcal{S}_{\Phi \mu_{\ell-1}}(u)} \big(\mathcal{F}_\ell(\mathcal{S}_{\Phi \mu_{\ell-1}}(u),\varepsilon_\ell)\big) \eqdist \mathcal{S}_{\Phi \mu_{\ell-1}}\circ\exp_{u} \big(\mathcal{F}_\ell(u,\varepsilon_\ell)\big)
\label{eq:equivarcondtransition}
\end{equation}
is true for all $u\in M$ in the domain of definition of $\log_{\Phi \mu_{\ell-1}}$. If $Z_{\ell-1}$ is reflection symmetric around $\mu_{\ell-1}$ then, $Z_\ell$ is reflection symmetric around $\Phi \mu_{\ell-1}$ and has $\Phi \mu_{\ell-1}$ as an exponential barycenter.
\end{theorem}

\begin{proof}
First we shall show that $p_{Z_{\ell}} (z)$ is reflection symmetric around $\Phi \mu_{\ell-1}$. Note that Lemma \ref{lem:reflsymmconservedunderiso} implies that $\Phi Z_{\ell-1}$ is reflection symmetric around $\Phi \mu_{\ell-1}$. By Bayes' theorem it holds that
\[
p_{Z_{\ell}}(z) = \int_{M} p_{Z_\ell\mid \Phi Z_{\ell-1}=v}(z)  p_{\Phi Z_{\ell-1}}(v) \dVol(v) = \int_M p_{\exp_{ v} (\mathcal{F}_\ell(v,\varepsilon_{\ell}))}(z) p_{\Phi Z_{\ell-1}}(v) \dVol(v).
\]
By \eqref{eq:equivarcondtransition} it holds that the probability density function satisfies
\begin{equation}
p_{\exp_{ \mathcal{S}_{\Phi \mu_{\ell-1}}u} (\mathcal{F}_\ell(\mathcal{S}_{\Phi \mu_{\ell-1}}u,\varepsilon_{\ell}))} (\mathcal{S}_{\Phi \mu_{\ell-1}}z)= p_{\mathcal{S}_{\Phi \mu_{\ell-1}}\exp_{ u}  \big( \mathcal{F}_\ell(u,\varepsilon_\ell) \big)} (\mathcal{S}_{\Phi \mu_{\ell-1}} z) = p_{\exp_{ u}  \big( \mathcal{F}_\ell(u,\varepsilon_\ell) \big)} (z)
\label{eq:transitionisdoublyreflection}
\end{equation}
since $\mathcal{S}_x\circ \mathcal{S}_x =\id$ and this holds for a.e. point $u$ outside the cut locus of $\exp_{\Phi \mu_{\ell-1}}$.
Let $U\subset  T_{\Phi \mu_{\ell-1}}M$ be the the maximal domain of injectivity for $\exp_{\Phi \mu_{\ell-1}}$. Since $M$ is connected and compact $\exp_{\Phi \mu_{\ell-1}}(U)$ covers $\dVol$-a.e. point on $M$. Moreover, since $M$ is in particular a naturally reductive space it is a D'Atri space and a change of coordinates $\exp_{\Phi \mu_{\ell-1}}(v) \mapsto\exp_{\Phi \mu_{\ell-1}}(-v) $ preserves the volume form, see \cite{d1975geodesic}. Therefore, for $w\in U$ and in local coordinates $v=\exp_{\Phi \mu_{\ell-1}}(u)$ utilizing a change in coordinates through the map $\mathcal{S}_{\Phi \mu_{\ell-1}}$ it holds that
\begin{multline*}
p_{Z_{\ell}}(\exp_{\Phi \mu_{\ell-1}}(-w)) = p_{Z_{\ell}}(\mathcal{S}_{\Phi \mu_{\ell-1}}\exp_{\Phi \mu_{\ell-1}}(w))\\
=\int_U  p_{\exp_{ \exp_{\Phi \mu_{\ell-1}}(u)} (\mathcal{F}_\ell(\exp_{\Phi \mu_{\ell-1}}(u),\varepsilon_{\ell}))}(\mathcal{S}_{\Phi \mu_{\ell-1}}\exp_{\Phi \mu_{\ell-1}}(w)) p_{\Phi Z_{\ell-1}}(\exp_{\Phi \mu_{\ell-1}}(u)) \dVol(u)\\
= \int_{-U} p_{\exp_{ \mathcal{S}_{\Phi \mu_{\ell-1}} \exp_{\Phi \mu_{\ell-1}}(u)} (\mathcal{F}_\ell(\mathcal{S}_{\Phi \mu_{\ell-1}}\exp_{\Phi \mu_{\ell-1}}(u),\varepsilon_{\ell}))}(\mathcal{S}_{\Phi \mu_{\ell-1}}\exp_{\Phi \mu_{\ell-1}}(w)) \\ \cdot p_{\Phi Z_{\ell-1}}(\mathcal{S}_{\Phi \mu_{\ell-1}}\exp_{\Phi \mu_{\ell-1}}(u)) \dVol(u)\\
=\int_Up_{\exp_{ \exp_{\Phi \mu_{\ell-1}}(u)} (\mathcal{F}_\ell(\exp_{\Phi \mu_{\ell-1}}(u),\varepsilon_{\ell}))}(\exp_{\Phi \mu_{\ell-1}}(w)) p_{\Phi Z_{\ell-1}}(\exp_{\Phi \mu_{\ell-1}}(u)) \dVol(u)\\
= p_{Z_{\ell}}(\exp_{\Phi \mu_{\ell-1}}(w))
\end{multline*}
where we have used Lemma \ref{lem:symmetric-injectivity-domain}, the identity \eqref{eq:transitionisdoublyreflection} and that $\Phi Z_{\ell}$ is reflection symmetric around $\Phi \mu_{\ell-1}$.

Next we need to show that the function $f:M\rightarrow [0,\infty)$ defined by
\[
f(y) = \int_M p_{Z_\ell}(z) \dist(z,y)^2 \dVol_M(z)
\]
has derivative $0$ when $y= \Phi \mu_{\ell-1}$. To see this, consider the gradient of $f$  w.r.t. $y$. By definition of the distance function, it holds that $\nabla_y\dist(y,z)^2 = -2\log_y(z)$. Compactness of $M$ justifies integrability under the integration sign and therefore
\[
\nabla f(y)= -2\int_M  \log_{y}(z) p_{Z_\ell}(z) \dVol_M(z)
\]
which will be zero at $y= \Phi \mu_{\ell-1}$ since $p_{Z_\ell}$ is reflection symmetric around $\Phi \mu_{\ell-1}$ and since $\log_{\Phi x}(z)$ is anti-symmetric around $\Phi \mu_{\ell-1}$. 
\end{proof}
}

Note that we have existence of a map $\mathcal{F}_\ell$ satisfying \eqref{eq:equivarcondtransition}. One can construct such a map by considering $\varepsilon_\ell$ isotropic with $\|\varepsilon_\ell\|$ less that the injectivity radius of $M$ almost surely. By isotropy, we can identify $\varepsilon_\ell$ with a tangent valued $G_\ell(u,\varepsilon_\ell)\in T_uM$ via any isometry from $\mathbb{R}^m$ to $T_uM$. Then, define a Bernoulli random variable $B_\ell: \Omega \rightarrow \{0,1\}$ with $\mathbb{P}(B_\ell)=\frac{1}{2}$ that is independent of all $\varepsilon_j$ for all $j\leq \ell$ and all $Z_j$ with $j\leq \ell-1$ then the almost surely defined map 
\[
\mathcal{F}_\ell(u,\varepsilon_\ell)=B_\ell G_\ell(u,\varepsilon_\ell) +(1-B_\ell) \log_{u}\circ \mathcal{S}_{\Phi \mu_{\ell-1}} \circ \exp_{\mathcal{S}_{\Phi \mu_{\ell-1}} u}(G_{\ell}(\mathcal{S}_{\Phi \mu_{\ell-1}} u,\varepsilon_\ell))
\]
satisfy the identity \eqref{eq:equivarcondtransition}.

The model formulation can be simplified in the case that the manifold is a $G$-symmetric space.

\begin{cor}
    Let $M$ be a $G$-symmetric space. Assume $\varepsilon_\ell$ is isotropic and that $\mathcal{F}_{\ell}(u,\varepsilon_\ell)$ is the simple identification of $\varepsilon_\ell$ via an isometry from $\mathbb{R}^m$ to $T_u M$.
Let
\[
Z_{\ell}= \exp_{\Phi Z_{\ell-1}}(\mathcal{F}_\ell(\Phi Z_{\ell-1},\varepsilon_\ell))
\]
be an AR(1)-process on $M$ as in \eqref{eq:autoregressivehomo}. Suppose that $\varepsilon_\ell$ is absolutely continuous on $\mathbb{R}^m$ and assume both $Z_{\ell-1},Z_\ell$ are both absolutely continuous on $M$. If $Z_{\ell-1}$ is reflection symmetric around $\mu_{\ell-1}$ then, $Z_\ell$ is reflection symmetric around $\Phi \mu_{\ell-1}$ and has $\Phi \mu_{\ell-1}$ as an exponential barycenter.
\end{cor}

\begin{proof}
Since $M$ is a symmetric space, $\mathcal{S}_{\Phi \mu_{\ell-1}}$ is an isometry on $M$. Since $\varepsilon_{\ell}$ isotropic $\exp_u(\mathcal{F}_\ell(u,\varepsilon_\ell))$ is in particular radially symmetric around $u$ for all $u\in M$ and
it holds that 
\[ 
p_{\exp_{\mathcal{S}_{\Phi x}v}(\varepsilon_\ell)} (\mathcal{S}_{\Phi x}z) = p_{\exp_v(\varepsilon_\ell)}(z)
\]
since $ \mathcal{S}_{\Phi \mu_{\ell-1}}$ preserves distances by isometry. Then the result follows from Theorem \ref{thm:ARreflectionsymmetric}
\end{proof}

{

Next we shall impose stronger assumptions on the distributions and the manifold itself in order to guarantee that the exponential barycenter in Theorem \ref{thm:ARreflectionsymmetric} is actually a Fr\'echet mean. In order to prove this stronger assertion, one might ask whether "adding" unbiased (in the sense of the Fr\'echet mean) noise to a prior, will the posterior distribution have the same Fr\'echet mean as one would get if one averaged out the noise? This is akin to asking if the tower property, $\mathbb{E}[\mathbb{E}[X\mid Y]]= \mathbb{E}[X]$ holds on a Rimannian manifold under the Fr\'echet mean. For manifolds of non-negative sectional curvature this assertion is generally false. By imposing strong restrictions on the prior and noise, the tower property is intimately related to the Kneser-Poulsen conjecture on Riemannian manifolds. This conjecture (originally on $\mathbb{R}^d$ by Kneser and Poulsen) has been solved partially on the generalization to Riemannian manifolds, see \cite{csikos2011volume,csikos2010extendability}, which inspire the following definition.

\begin{definition}
\label{def:strongKP2}
    Let $M$ be a Riemannian manifold, and let $B_{r}(x)$ denote a geodesic ball on $M$ of radius $r$ and center $x$. The manifold $M$ is said to satisfy the \emph{strong $\mathrm{KP}_2$ property} if,
    \[
    \dist(x_1,x_2)\leq \dist(y_1,y_2) \Rightarrow \Vol (B_{r_1}(x_1)\cap B_{r_2}(x_2)) \geq \Vol(B_{r_1}(y_1)\cap B_{r_2}(y_2))
    \]
\end{definition}

\begin{remark}
    Note here that we call this the strong $\mathrm{KP}_2$ property to distinguish from its weaker $\mathrm{KP}_2$ property counterpart introduced in \cite{csikos2010extendability}. This weaker form states that $\Vol (B_{r_1}(x_1)\cap B_{r_2}(x_2))$ is a function depending only on $r_1,r_2$ and $\dist(x_1,x_2)$ (and not on the individual points $x_1,x_2$). This weaker form has been shown to require that the Riemannian manifold is an Einstein manifold \cite[Theorem 4]{csikos2011volume} and the manifold is simply connected it is harmonic \cite[Theorem 3]{csikos2011volume}. Unfortunately, this type of requirement does then preclude the Stiefel manifolds.
\end{remark}

We shall record a theorem due to McCormack and Hoff, see \cite[Theorem 2]{mccormack2023equivariant}.

\begin{theorem}
\label{thm:radiallydecreasinghasFrechetmean}
    Let $X:\Omega \rightarrow M$ be an absolutely continuous random variable onto a homogeneous Riemannian space $M$. Suppose that the probability density function of $X$ is radially symmetric around $\mu \in M$ and that it is radially decreasing away from $\mu$. Then $X$ has Fr\'echet mean $\mu$ and if $X$ is strictly radially decreasing, $\mu$ is the unique Fr\'echet mean.
\end{theorem}

\begin{lemma}
\label{lem:radiallydecreasingconvolution}
    Let $X,Y:\Omega \rightarrow M$ be radially symmetric, absolutely continuous random variables on a compact homogeneous Riemannian space $M$ satisfying the strong $\mathrm{KP}_2$ property and let $\varepsilon: \Omega \rightarrow T_x M$ be a radially symmetric absolutely continuous random variable independent of $X$ for every $x\in M$. Suppose the random variable $Y$ is given by
    \[
    Y= \exp_X(\varepsilon).
    \]
    If $X$ is radially decreasing away from $\mu = \mathbb{E}[X]$ and if the probability density function of $Y$ conditional to $X$ is radially decreasing away from $X$, then $Y$ is radially decreasing away from $\mu$.
\end{lemma}

\begin{remark}
    Note that Theorem \ref{thm:radiallydecreasinghasFrechetmean} applies to $X$ and $\mu$ is truly a Fr\'echet mean and presupposing $\mu = \mathbb{E}[X]$ is justified.
\end{remark}

\begin{proof}
Note that by Bayes' theorem it holds that
\[
p_{Y}(y) = \int_M p_{Y\mid X=u}(y) p_{X}(u) \dVol(u).
\]
By assumption $p_{Y\mid X=u}(y)$ is radially decreasing away from $u$ and $p_X$ is radially decreasing away from $\mu$. Let $B_r(x)$ denote the geodesic ball
\[
B_r(x)= \{ z\in M :\dist(x,z) <r\}.
\]
The main strategy is to approximate $p_{Y\mid X=u}$ and $p_X$ by simple functions and pass through the limit by the dominated convergence theorem. Let $\mathds{1}_A$ denote the indicator function for the set $A\subseteq M$. Note that $\mathds{1}_{B_{r_1}(u)}(x)$ is radially decreasing away from $u$ and that $\mathds{1}_{B_{r_2}(\mu)}$ is radially decreasing away from $\mu$. It holds that the convolution of these indicator functions can be computed as follows
\[
\int_M \mathds{1}_{B_{r_1}(u)}(x) \mathds{1}_{B_{r_2}(\mu)}(u) = \Vol(B_{r_1}(x) \cap B_{r_2}(\mu)).
\]
Since $M$ satisfies the strong $\mathrm{KP_2}$ property the right hand side is radially decreasing and is only depending on $\dist(\mu,x)$ and $r_1,r_2$ and is therefore radially symmetric around $\mu$ as well. Let 
\[
L=\diam(M)= \sup \{\dist(x,y): x,y\in M\}
\]
denote the diameter of $M$, which exists since $M$ is compact and let $\partial B_\rho(\mu)= \{x\in M:\dist(x,\mu)=\rho\}$ denote the geodesic sphere of radius $\rho$ centered at $\mu$. Note that since $p_X$ and $p_{Y\mid X}$ are radially symmetric around $\mu$ and $X$, respectively, they are well defined on $\partial B_\rho(\mu)$ and $\partial B_\rho(X)$, respectively. We shall use the convention that $p_X(\partial B_{\rho}(\mu))= p_{Y\mid X=u}(\partial B_\rho (u)) =0$ in the case that $\rho> L$. Now consider the sequence of simple functions $f_n:M \rightarrow [0,\infty)$ defined by
\[
f_n(u) = \sum_{j=1}^n \mathds{1}_{B_{\frac{jL}{n} }(\mu)}(u) \big(p_X(\partial B_{\frac{jL}{n}}(\mu))- p_X(\partial B_{\frac{(j+1)L}{n}}(\mu)) \big)
\]
it holds that $\lim_{n\to \infty}f_n(u)= p_X(u)$. Similarly, consider the sequence of simple functions $g_m:M\rightarrow [0,\infty)$, for each fixed $u\in M$, defined by
\[
g_m(x)=  \sum_{i=1}^m \mathds{1}_{B_{\frac{iL}{m} }(u)}(x) \big(p_{Y\mid X= u}(\partial B_{\frac{iL}{m}}(u))- p_{Y\mid X= u}(\partial B_{\frac{(i+1)L}{m}}(u)) \big).
\]
it holds that $\lim_{m\to \infty} g_m(x)=p_{Y\mid X=u}(x)$. Note then that
\begin{multline*}
h_{n}(x) = \int_{M} g_n(u,x) f_n(u) \dVol(u) \\
= \sum_{j=1}^n \sum_{i=1}^n \big(p_X(\partial B_{\frac{jL}{n}}(\mu))- p_X(\partial B_{\frac{(j+1)L}{n}}(\mu)) \big) \\
\big(p_{Y\mid X= \mu}(\partial B_{\frac{iL}{n}}(\mu))- p_{Y\mid X= \mu}(\partial B_{\frac{(i+1)L}{n}}(\mu)) \big) \Vol(B_{\frac{jL}{n} }(\mu)\cap B_{\frac{iL}{n} }(x)).
\end{multline*}

We claim that $h_n(x)$ decrease as $\dist(\mu, x)$ increase. Fix $x$ such that $\frac{iL}{n}\geq \dist(\mu,x)< \frac{(i+1)L}{n}$ for fixed $0\leq i\leq n-1$ and if $y\in M$ is such that $\frac{iL}{n}\geq \dist(\mu,y)< \frac{(i+1)L}{n}$ and $\dist(\mu,x)< \dist(\mu,y)$ then $h_n(x)\geq h_n(y)$ since the strong $\mathrm{KP_2}$-property holds true. If instead $\dist(\mu,y) > \frac{(i+1)L}{n}$ then $h_n(x) \geq h_n(y)$ since the intersected geodesic balls are smaller and fewer terms in the sum are non-zero. The dominated convergence theorem holds and $\lim_{n\to \infty} h_n(x) = p_Y(x)$. For each $n$, $h_n$ is radially decreasing away from $\mu$ and radially symmetric around $\mu$, and so is $p_Y$.
\end{proof}

} Now we are ready to prove the {second} main theorem.

{

\begin{theorem}
\label{thm:ARexpect}
Let $M$ be a connected $m$-dimensional $G$-homogeneous space with $G$ a compact Riemannian Lie group with bi-invariant metric and with Riemannian metric on $M$ inherited from $G$. Assume $M$ has the strong $\mathrm{KP_2}$-property. Let $\varepsilon_\ell$ is isotropic and identify $\mathcal{F}_{\ell}(u,\varepsilon_\ell)$ as the identification of $\varepsilon_\ell$ onto the tangent space $T_u M$ via an isometry from $\mathbb{R}^m$ to $T_u M$.
Let
\[
Z_{\ell}= \exp_{\Phi Z_{\ell-1}}(\mathcal{F}_\ell(\Phi Z_{\ell-1},\varepsilon_\ell))
\]
be an AR(1)-process on $M$ as in \eqref{eq:autoregressivehomo}.
Suppose that $\varepsilon_\ell$ is absolutely continuous on $\mathbb{R}^m$ and $Z_{\ell-1},Z_\ell$ are absolutely continuous on $M$ with respective unique Fr\'echet means $\mathbb{E}[\varepsilon_\ell]=0$ and $\mathbb{E}[Z_{\ell-1}]$. Moreover, assume that  $\exp_x(\mathcal{F}_{\ell}(x,\varepsilon_\ell))$ is radially decreasing away from $x$, for every fixed $x$. Similarly, assume $Z_{\ell -1}$ is radially symmetric and radially decreasing away from $\mathbb{E}[Z_{\ell-1}]$. Then, $Z_\ell$ is radially decreasing away from $\Phi\mathbb{E}[Z_{\ell-1}]$ and
\[
\mathbb{E}[Z_\ell]= \Phi \mathbb{E} [Z_{\ell-1}].
\]
\end{theorem}

\begin{proof}
Set $x=\mathbb{E}[Z_{\ell-1}]$. Note that since $\Phi\in G$ is an isometry of $M$ the random variable $\Phi Z_{\ell-1}$ will be radially symmetric and radially decreasing away from $\Phi x$. To see this, notice that 
\[
p_{\Phi Z_{\ell-1} }(\partial B_r(\Phi x)) = p_{Z_{\ell-1}}(\Phi^{-1}\partial B_r(\Phi x)) 
\]
where we have the Jacobian determinant of an isometric transformation has absolute value 1. Since $\Phi$ is an isometry on a homogeneous space $\Phi^{-1} \partial B_r(\Phi x)= \partial B_r(x)$. Hence $\Phi Z_{\ell-1}$ is radially symmetric and radially decreasing away from $\Phi x$. Now Lemma \ref{lem:radiallydecreasingconvolution} applies to 
\[
Z_{\ell} = \exp_{\Phi Z_{\ell-1}}(\varepsilon_\ell)
\]
and $Z_{\ell}$ is radially symmetric and radially decreasing away from $\Phi x$. Immediately Theorem \ref{thm:radiallydecreasinghasFrechetmean} applies and 
\[
\mathbb{E}[Z_{\ell}]= \Phi \mathbb{E}[Z_{\ell-1}].
\]
\end{proof}
}

\section{Conjugate Gradient Methods to Solve System Identification}
\label{sec:conjgrad}

Here we propose a way of conjugate gradient descent on the group action representation of autoregressive processes on the Orthogonal/Unitary group and some homogeneous spaces associated with $O(n)$ and $U(n)$ respectively. In order to simplify notation we shall only write down the formulas for the real case. All formulas can easily be adapted to the complex case by replacing the (real) transpose of a matrix $X$ with the complex conjugate transpose $X^*$ whenever the (real) transpose $X^T$ appears. Note that inherent to this formulation of AR-processes, the high degree of multicollinearity of the parameters will make algorithms such as Newton's method impractical since the second derivative of the M.S.E. will be unstable when noise is small.

\subsection{Orthogonal/Unitary groups}
\label{sec:orth}

Here we shall consider the special case when the manifold is the Lie group itself, i.e. we have observations $Y_\ell$ following \eqref{eq:autoregressivehomo} on $O(n)$. This part is solved by previous works, see \cite{fiori2009algorithm}, but we state it here for completeness. Due to the group structure of the manifold, the parameter of an AR(1) process may immediately be estimated as
\[
\hat{\Phi}_{\ell}=Y_{\ell+1} Y_{\ell}^{-1},
\]
for every $\ell$. Then the average $\hat{\Phi}$ may be computed using a barycenter algorithm on all the $\hat{\Phi}_{\ell}$ as in \cite{pennec2013exponential}. Numerical experiments exhibit the expected behavior as seen in Figure \ref{fig:Liegroupaverage}.

\subsection{Stiefel Manifolds}
\label{sec:stief}

Consider  $M=\St_{n,k}$, the real compact Stiefel manifold with the natural left $SO(n)$-action. We want to minimize the  {objective function} $f:O(n)\rightarrow \mathbb{R}_{\geq 0}$ defined by
\[
f(\Phi) = \sum_{\ell=1}^N \dist(Y_{\ell+1},\Phi Y_{\ell})^2,
\]
where $Y_\ell$ is observations of the AR(1)-process \eqref{eq:autoregressivehomo}. The issue with solving the minimization problem on Stiefel manifolds for a general $f$ using a Newton algorithm has been treated in \cite{edelman1998geometry}, though this requires a computation of the Euclidean gradient and Hessian of $f$.

For the induced metric (aka canonical metric) on $M$ the geodesic exponential map corresponds to
\[
\exp_X\begin{pmatrix}
A\\ B
\end{pmatrix}= \exp_{\mathrm{M}} \begin{pmatrix}
A & -B^T\\ B& 0
\end{pmatrix} X
\]
where $\exp_{\mathrm{M}}$ denotes the matrix exponential. In contrast to the geodesic exponential map on $O(n)$ the corresponding inverse map is not easy to find. On the other hand we shall make use of the pad\'e approximation of the exponential map to get an extrinsic expression for $f$. Note that the Cayley transform
\[
\exp_X(\Delta)=\exp_{\mathrm{M}}(\pi_{Horizontal}\Delta X^T)X\\ \approx \left(I_n-\frac{1}{2}\Delta X^T \right)^{-1} \left(I_n+\frac{1}{2}\Delta X^T\right)X,    
\]
approximates the Riemannian exponential map see \cite{jauch2020random,yuan2019global,zhu2019matrix}, and thus given $X,Y\in M$ by straightforward algebra it holds that the inverse of the Cayley transform can be computed by
\[
\log_X(Y)\approx 2(Y-X) (I_k +X^TY)^{-1}.
\]
\begin{remark}
There is a method developed by Ralf Zimmermann, see \cite{zimmermann2017matrix}, which approximates the logarithm map more accurately. The proposed method here could be modified in this regard, though we believe this second order approximation is sufficient for our purposes.
\end{remark}
Since 
\[
\dist(X,Y)^2= \norm{\log_X(Y)}^2\\ \approx \frac{1}{2} \tr(\log_X(Y)^T \left(I_k-\frac{1}{2} XX^T\right)\log_X(Y))    
\]
we have the following approximation of $f$,
\begin{multline*}
f(\Phi) \approx \sum_{\ell=1}^N \tr \Big(\left(I_k+Y^T_{\ell}\Phi^T Y_{\ell+1}\right)^{-1}\Big( 3I_k - Y_\ell^T \Phi^TY_{\ell+1}-Y_{\ell+1}^T \Phi Y_\ell\\ 
  - Y_\ell^T\Phi^T Y_{\ell+1} Y_{\ell+1}^T \Phi Y_\ell\Big) \left( I_k +Y_{\ell+1}^T \Phi Y_\ell\right)^{-1} \Big).    
\end{multline*}

To simplify notation we shall denote
\[
A=\left(I_k + Y_\ell^T \Phi^T Y_{\ell+1}\right)^{-1},
\]
\[
B=3I_k -Y_{\ell+1}^T\Phi Y_\ell - Y_\ell^T \Phi^T Y_{\ell+1} \\- Y_\ell^T\Phi Y_{\ell+1} Y_{\ell+1}^T \Phi Y_\ell,    
\]
and
\[
C= \left(I_k + Y_{\ell+1}^T \Phi Y_{\ell}\right)^{-1}
\]
Consider the inversion map $\inv X\mapsto X^{-1}$ on $\mathrm{GL}(n)$, it holds that
\[
D\inv_X(\Delta) = -X^{-1} \Delta X^{-1}.
\]
The Euclidean derivative of $f$ w.r.t. $\Phi$ is the following linear map
\begin{multline*}
\nabla f_\Phi(\Delta)= -2\sum_{\ell=1}^N \tr \Big(A  Y_\ell^T \Delta^T Y_{\ell+1} AB C + A \big(  Y_\ell^T \Delta^T Y_{\ell+1} +  Y_{\ell+1}^T \Delta Y_{\ell}\\
+Y_\ell^T\Delta^T Y_{\ell+1} Y_{\ell+1}^T \Phi Y_{\ell} + Y_\ell^T \Phi^T Y_{\ell+1} Y_{\ell+1}^T \Delta Y_{\ell}\big)C
 +ABCY_{\ell+1}^T \Delta Y_{\ell}C\Big).    
\end{multline*}

Thus, we have the following scheme for estimating $\Phi\in SO(n)$, using conjugate gradient descent, given $\{Y_\ell\}_{\ell=1}^{N+1}\subset \St_{n,k}$ and initial guess $\Phi_0$.

\begin{algorithm}
\caption{Estimating system parameters on $\St_{n,k}$ by conjugate gradient descent.}
\label{alg:stiefelgrad}
\begin{algorithmic}[1]
\State Compute 
\[
\Delta=\sum_{j}^{\dim (O(n))}\nabla f_{\Phi_0}(\Phi_0 e_j) \cdot e_j
\]
where $e_j$ is an ON basis on $\mathfrak{so}(n)$.
\State Find $\tau\in [-1,1]$ such that
\[
f(\Phi_0 \exp(-\tau \Delta))
\]
is minimal.
\State Update $\Phi$ by
\[
\Phi= \Phi_0\exp(-\tau \Delta)
\]
\State Repeat step 1-3 with $\Phi_0=\Phi$ but with $\Delta$ orthogonal to previous $\Delta$ using a Gram-Schmidt procedure, repeat this for a prescribed number of iterations or until $\text{steps}=\dim(O(n))$.
\State Repeat step 1-4 with $\Phi_0=\Phi$ until gradient step is sufficiently small, e.g. $\norm{\Delta} \abs{\tau}<Tol$.
\end{algorithmic}
\end{algorithm}

\begin{remark}
    Note that step 1 in Algorithm \ref{alg:stiefelgrad} can be computed utilizing modern algorithms such as those from the Manopt package.
\end{remark}

Figures \ref{fig:graddescentst},\ref{fig:graddescentnoise},\ref{fig:graddescentstoverd},\ref{fig:graddescentstoverk} show some simulations estimating the system parameter of an AR(1) process on the Stiefel manifolds.

\subsection{Grassmann Manifolds}
\label{sec:grass}

Here we consider the Grassmann manifold $M=\Gr_{n,k}$, we identify $\Gr_{n,k}$ as equivalence classes of $n\times k$ matrices by $\St_{n,k}/O(k)$. Once again given observations $\{Y_\ell\}_{\ell=1}^{N+1}\subset M$ we want to find $\Phi \in SO(n)$ such that
\[
\sum_{\ell=1}^N \dist(Y_{\ell+1},\Phi Y_{\ell})^2 \leq \sum_{\ell=1}^N \dist(Y_{\ell+1},\Psi Y_{\ell})^2
\]
for all $\Psi  \in SO(n)$.
At $[I_{n,k}]$ the tangent space can be identified as
\[
T_{I_{n,k}}\Gr_{n,k}=\left\{\begin{pmatrix}
0\\ V
\end{pmatrix}:V\in \mathcal{M}_{n-k,k}(\mathbb{R})\right\}.
\]
In general it is very difficult to write out an explicit parametrizable tangent vector. The question is therefore: Given $\Delta\in T_Y \St_{n,k}$ how do we get $[\Delta]\in T_{[Y]}\Gr_{n,k}$ in order to compute $\dist([X],[Y])$? A general tangent vector in $T_{Y}\St_{n,k}$ may be written as
\[
\Delta = \begin{pmatrix}
A& -B^T\\B& 0
\end{pmatrix} X\qquad A\in \mathfrak{o}(k), B\in \mathcal{M}_{n,k}(\mathbb{R})
\]
and the corresponding Grassmannian tangent space (i.e. in horizontal direction) is
\[
[\Delta] = \begin{pmatrix}
0& -B^T\\B& 0
\end{pmatrix} X\qquad B\in \mathcal{M}_{n,k}(\mathbb{R}).
\]
Now since $X^T \Delta=A$ is anti-symmetric we have the following Grassmann metric at the base point $X\in \St_{n,k}$
\[
g_X(V,W)^2=\tr(V^T(I_n-XX^T)W),
\]
note the difference of a factor of $\frac{1}{2}$ compared to the Stiefel case. Moreover, note that this metric is degenerate on the Stiefel manifold but non-degenerate precisely on the Grassmann manifold. Therefore the map we want to minimize can be approximated as follows:
\[
f(\Phi)=\sum_{\ell=1}^N 4 \tr\left(A (I_k-Y_{\ell}^T\Phi^T Y_{\ell+1}Y_{\ell+1}^T \Phi Y_{\ell})C \right)
\]
where $A$ and $C$ are as previously defined. Let
\[
D=I_k-Y_{\ell}^T\Phi^T Y_{\ell+1}Y_{\ell+1}^T \Phi Y_{\ell}.
\] 
Then, the gradient, with respect to $\Phi$, is

\begin{multline*}
\nabla f_\Phi(\Delta)= -\sum_{\ell=1}^N 4 \tr \Big(A  Y_\ell^T \Delta^T Y_{\ell+1} AD C\\
+ A \big( Y_\ell^T\Delta^T Y_{\ell+1} Y_{\ell+1}^T \Phi Y_{\ell} 
 + Y_\ell^T \Phi^T Y_{\ell+1} Y_{\ell+1}^T \Delta Y_{\ell}\big)C +ADCY_{\ell+1}^T \Delta Y_{\ell}C\Big).
\end{multline*}

Now the algorithm is precisely as done in the case of the Stiefel manifold, see Algorithm \ref{alg:stiefelgrad}, but with the above gradient instead. In Figures \ref{fig:graddescentgr},\ref{fig:graddescentgrnoise},\ref{fig:graddescentgroverd},\ref{fig:graddescentgroverk} one can find some simulations estimating the parameters of an AR(1) process using this algorithm.

\begin{appendices}

\section{Basic Riemannian Geometry}
\label{sec:basicriemgeom}

To make this paper more accessible, we shall here give a brief introduction to some concepts in differential geometry that are necessary to understand the results of this paper. Everything discussed in this section can be found in \cite{lee2009manifolds} which gives a good modern summary of differential geometry.

The most fundamental concept is that of a manifold, its aim is to generalize $\mathbb{R}^n$ to mathematical objects that locally behave like $\mathbb{R}^n$, a framework which allows for analysis (both stochastic and classical) on them. Let $M$ be a set. A \textit{chart} on $M$ is a subset $U_\alpha\subseteq M$ and an injective map $u_{\alpha}:U\rightarrow \mathbb{R}^n$, often denoted as $( U_\alpha,u_\alpha)$. Two charts, $(U_\alpha,u_\alpha)$ and $(U_\beta,v_\beta)$ are said to be (smoothly) compatible either if, $U_\alpha \cap U_\beta = \emptyset$, or if $U_\alpha \cap U_\beta \neq \emptyset$, then the map
\[
u_\alpha \circ (v_\beta\mid_{U_\alpha \cap U_\beta})^{-1} :v_\beta(U_\alpha \cap U_\beta) \rightarrow u_\alpha(U_\alpha \cap U_\beta)
\]
is a diffeomorphism between open subsets in $\mathbb{R}^n$.  An \textit{atlas} is a collection of compatible charts $\mathcal{A}=\{( U_\alpha,u_\alpha)\}_{\alpha \in \mathscr{A}}$, with $\mathscr{A}$ some indexing set, such that
\[
\bigcup_{\alpha \in \mathscr{A}} U_\alpha = M.
\]
Moreover, it is a classical exercise to show that compatibility is an equivalence relation on atlases. The charts allow for differentiation of functions between manifolds. Consider $f:M\rightarrow N$, where $M$ and $N$ are smooth manifolds. Fix $p\in M$. Let $u:U\rightarrow \mathbb{R}^m$ be a chart around $p$, let $v:V\rightarrow \mathbb{R}^n$ be a chart around $f(p)$ and suppose $f(U)\subset V$. Then, $f$ is called \textit{smooth} at $p$ if 
\[
v\circ f \circ u^{-1}:u(U) \rightarrow v(V)
\]
is a smooth function at $u(p)$ in the sense of functions between finite dimensional vector spaces. Next, we have the tangent space.

\begin{definition}
Let $M$ be a smooth manifold and let $p\in M$ be a point. Denote $C_p$ as the set of all curves $\gamma:(-\varepsilon,\varepsilon): \rightarrow M$ such that $\gamma(0)=p$. We say the two curves $\gamma_1, \gamma_2\in C_p$ are equivalent, $\gamma_1 \sim \gamma_2$ if for all smooth $f:M\rightarrow \mathbb{R}$, it holds that
\[
\ddf { }{ }{t} f\circ \gamma_1 \mid_{t=0} = \ddf { }{ }{t} f\circ \gamma_2 \mid_{t=0}.
\]
Then, the \textit{tangent space} at $p$ is the vector space
\[
T_p M := C_p/\sim.
\]
The \textit{tangent bundle} $TM$ is the disjoint union of all the tangent spaces, i.e.
\[
TM = \bigsqcup_{p\in M} T_pM
\]
\end{definition}

Moreover, the \textit{tangent map} of $f:M\rightarrow N$ at $p$, is the map $D_p f:T_pM \rightarrow T_{f(p)}M$ defined by 
\[
\gamma\in C_p \longmapsto f\circ \gamma \in C_{f(p)}.
\]
It is easy to check that this is a well-defined assignment.

In order to solve extremal problems on manifolds, one has the need of (topological) metrics. In differential geometry the topological metric on a manifold is called the \textit{distance}, $\dist$, and it is generated by the so called \textit{Riemannian metric}. For every $p\in M$, the Riemannian metric at $p$ is the bi-linear map
\[
g_p : T_pM \times T_pM\rightarrow \mathbb{R}
\]
such that $g_p$ is positive definite and symmetric. Furthermore, as we vary the base point $p$ this assignment of linear maps should be smooth, this means that 
\[
g:p\longmapsto g_p(T_{u(p)}u^{-1} v,T_{\tilde{u}(p)}\tilde{u}^{-1}w), \qquad v,w\in \mathbb{R}^n
\]
is a smooth assignment where $u,\tilde{u}:U\rightarrow \mathbb{R}^n$ are charts around $p$. If $\gamma:[a,b]\rightarrow M$ is a curve on $M$ then, the distance between $p=\gamma(a)$, $q=\gamma(b)$ along $\gamma$ is defined as
\[
\dist_\gamma (p,q)=\int_a^b g_{\gamma(s)}(T_s\gamma,T_s\gamma)^{1/2} ds.
\]
The \textit{geodesic} curve going from $p$ to $q$ is any curve $\tau:[a,b]\rightarrow M$ such that $\dist_\tau(p,q) \leq\dist_\gamma(p,q)$ for all curves $\gamma$ starting at $p$ and ending at $q$. Moreover, the \textit{distance} $\dist(p,q)$ between two points $p,q\in M$ is defined as the distance along the geodesic going from $p$ to $q$. Observe here that the distance function is a topological metric on $M$. Note, that $\tau$ is only guaranteed to be unique when $p$ and $q$ lie in a neighborhood small enough. Since geodesics are locally unique one can define a chart which is locally an isometry between $T_p M$, with metric from the inner product $g_p$ and the distance function on $M$.

\begin{definition}
\label{def:geodexplog}
Let $p\in M$ be a point and let $U\subseteq T_pM$ be open and such that all length minimizing geodesics $[0,\norm{V}]\rightarrow M$ with initial velocity $X= V/\norm{V}$, $V\in U$ are unique. Then, the assignment $\exp_p(V) = \gamma_p(\|V\|_p)$ where $\gamma_p$ is the unique geodesic starting at $p$ with initial velocity $V/\norm{V}$ and length $\norm{V}$, is a map $U\rightarrow M$ called the \textit{Riemannian exponential map} at $p$. Its inverse $\log_p: \exp_p(U)\rightarrow U$ is the \textit{geodesic normal chart} around $p$, and it's associated coordinates are called the \textit{geodesic normal coordinates}. The maximal $\rho_p >0$ such that $\exp_p$ is injective when it is restricted to the ball
\[
B_\rho(0) = \{x\in T_p X: \norm{x}<\rho\}
\]
is referred to as the \textit{injectivity radius of $M$ at $p$}. The minimal injectivity radius at all points on $M$, $\rho =\inf_{p\in M} \rho_p$ is referred to as the \textit{injectivity radius of $M$}.
\end{definition}

The geodesic normal coordinates are of particular importance. In geodesic normal coordinates, the gradient and Hessian are reduced to the Euclidean gradient and Hessian. That is, if $x_i$ are geodesic normal coordinates centered at $p$, i.e. the functions generated by the exponential map and some orthonormal basis on $T_pM$, then 
\[
\nabla f(p) = \frac{\partial f}{\partial x_i}(p) \partial_{x_i}
\]
and
\[
\Delta f(p) = \frac{\partial^2f}{\partial x_i^2} (p).
\]
Utilizing spherical geodesic normal coordinates centered around $p$, 
\[(r,\theta_1,\dots , \theta_{n-1}),\]
i.e. the transformed geodesic normal coordinates $x_i$ pre-composed with the transformation that defines the spherical coordinates in $\mathbb{R}^n$, one has $r=\dist(p,q)$. Moreover, we shall mean by the coordinate $(-r,\theta_1,\dots, \theta_{n-1})$ the reflected vector that lies in the opposite direction of $(r,\theta_1,\dots,\theta_{n-1})$. Moreover, manifolds for which the Riemannian exponential map is defined on all of $T_p M$ is referred to as \textit{geodesically complete}.

A Lie group $G$ is a smooth manifold and a group such that the group multiplication map
\[
\mu:(x,y) \longmapsto x\cdot_G y,
\]
and the inversion map
\[
\inv : x\longmapsto x^{-1}
\]
are smooth maps. The left action $L_p:G\rightarrow G$ is defined by $L_p(q)=\mu(p,q)$ and the right action $R_p:G \rightarrow G$ is defined by $R_p(q)=\mu(q,p)$. Sometimes, there exists a bi-invariant Riemannian metric on $G$, i.e. $g_p$ is such that
\[
g_{p}(v_p,w_p)=g_{\mu(q,p)}(T_p L_q v_p,T_pL_q w_p)=g_{\mu(p,q)}(T_p R_q v_p,T_pR_q w_p).    
\]
A very important special case are the matrix Lie groups, which are the Lie groups that are closed subgroups of invertible linear maps $\mathrm{GL}(n)$ over $\mathbb{R}$ or $\mathbb{C}$. In particular, the Orthogonal group $O(n)$ is the real (Matrix) Lie group defined as the set
\[
O(n)=\left\{ A\in \mathrm{GL}(n): AA^{T}= I_n \right\}
\]
together with the multiplication map and inversion map inherited from $\mathrm{GL}(n)$. The tangent space of $O(n)$ at the identity is precisely the Lie algebra $\mathfrak{o}(n)$ which is the space of alternating $n\times n$ matrices. On $O(n)$ we have the invariant bi-invariant metric, also known as the \textit{Frobenius inner product}
\[
g_{e}(X,Y) = \tr (X^T Y),
\]
for $X,Y\in \mathfrak{o}(n)$, then it is extended to all other points through the bi-invariance identity. For this metric, we have that the Riemannian exponential map coincides with the matrix exponential,
\[
\exp_{\mathrm{M}}(X) := \sum_{k=0}^\infty \frac{1}{k!} X^k.
\]
Note that we still distinguish between the two maps by designating the subscript $M$ for the matrix exponential map. 

Finally, one may consider a closed Lie subgroup $H$ of some Lie group $G$ and consider the quotient $G/H$. In general it is not a Lie group, since $H$ may not be normal. But it still has a lot of useful structure from the underlying Lie group. Such quotients are referred to as \textit{homogeneous spaces}. If $G$ is a Lie group with bi-invariant Riemannian metric $g$ one defines the \textit{induced metric} on the homogeneous space $G/H$ as restricting the metric to the horizontal part of the tangent space. More precisely since $\mathfrak{h}\subset \mathfrak{g}$ as a vector space we may (orthogonally) decompose any element $X\in \mathfrak{g}$ into horizontal part $\pi(X) \in T_{[e]}G/H$ and a vertical part $V(X)\in \mathfrak{h}$. In particular, by considering $O(n-k)$ as a subgroup of $O(n)$, $k<n$, through the embedding $\iota: O(n-k)\rightarrow O(n)$ defined by
\[
\iota: X \longmapsto \begin{pmatrix}
I_{k} & 0\\ 0 & X
\end{pmatrix}.
\]
One considers the quotient space $\St_{n,k} := O(n)/O(k)$ called the Stiefel manifold and the quotient space $\Gr_{n,k}:=\St_{n,k}/O(k)$ called the Grassmann manifold. Note that here $O(k)$ is considered as the embedded subgroup in $O(n)$ by
\[
X\longmapsto \begin{pmatrix}
X & 0\\ 0 & I_{n-k}.
\end{pmatrix}
\]
Observe that one may interpret the Stiefel manifold as the space of orthogonal $n$ by $k$ matrices and the tangent space of $\St_{n,k}$ at $I_{n,k}$ as the space of $n$ by $k$ matrices of the form
\[
T_{I_{n,k}} \St_{n,k} = \Bigg\{ \begin{pmatrix}
A\\ B
\end{pmatrix} \in \mathcal{M}_{n,k}(\mathbb{R}): A\in \mathfrak{o}(k) \text{ and } B\in \mathcal{M}_{n-k,k} \text{ arbitrary} \Bigg\}.  
\]
By considering the induced metric on $\St_{n,k}$ from the underlying metric on $O(n)$ we get a Riemannian structure on $\St_{n,k}$ and the Riemannian exponential map on $\St_{n,k}$ at $I_{n,k}$ is the map
\[
\exp_{I_{n,k}}\left(\begin{pmatrix}
A\\B
\end{pmatrix}\right) = \exp_{\mathrm{M}}\left( \begin{pmatrix}
A &-B^T\\ B & 0
\end{pmatrix}\right)
\]
and the exponential map at some arbitrary point $p\in \St_{n,k}$ is attained by post-multiplication from the right by $p$. Since there are infinitely many representations of $[X]\in \St_{n,k}$ as the equivalence class of matrices $X\in \mathfrak{o}(n)$, the choice of alternating matrix above is referred to as the horizontal projection
\[
\pi_{\text{Horizontal}}: \begin{pmatrix}
A & -B^T\\ B& C
\end{pmatrix} \longmapsto \begin{pmatrix}
A & -B^T\\ B& 0
\end{pmatrix}
\]
where $A\in \mathfrak{o}(k)$, $B\in \mathcal{M}_{n-k,k}$ and $C\in \mathfrak{o}(n-k)$.

\section{Simulations}
\label{sec:simulations}

Here we shall give an outline on some performed simulations for the respective algorithm. All the simulations here are for some special cases of $O(n)$-homogeneous spaces $M$. Moreover, for all cases, we start by  choosing some random element $\Phi \in SO(n)$ with standard deviation of $\dist(\Phi,I_n)\sim 0.01$. The noise terms are generated $N$-times by considering $\varepsilon_\ell$ as normal distribution on $\mathfrak{o}(n)$ with mean $0\in \mathfrak{o}(n)$ and small standard deviation $\sigma \id_{\mathfrak{o}(n)}$. This may be done by taking the normal distribution on $\mathcal{M}_{n,n}$ and then projecting to the subspace of anti-symmetric matrices. Then starting at some point $p\in M$, the elements $Z_{\ell}=\exp_{\mathrm{M}}(\varepsilon_\ell)\Phi Z_{\ell-1}$ are inductively generated with $Z_{0}=p$. Even though this is not exactly analogous to the formulation in Equation \eqref{eq:autoregressivehomo}, by Baker-Campbell-Hausdorff this is a first order approximation of the latter. Now we apply an algorithm for each case given in Section \ref{sec:conjgrad} which gives an estimate of $\hat{\Phi}$ and then the error is computed by $\dist(\Phi,\hat{\Phi})$.

For the case when $M=O(n)$ we apply the algorithm as outlined in Section \ref{sec:orth}, i.e. through the barycenter algorithm compute the average of all the $Z_{\ell+1}^{-1} Z_\ell$. Simulations suggest that the error between the estimate and $\Phi$ is of the order $\mathcal{O}(\sqrt{N}^{-1})$ with coefficient depending on all other parameters, $n,\Phi,\sigma$. This is expected since this algorithm is a generalization of sample average estimates in Euclidean space, which also has the same order term. In figure \ref{fig:Liegroupaverage} this is illustrated for the case $O(20)$ and $\sigma=0.1$.

\begin{figure}[H]
\centering
\captionsetup{width=.8\linewidth}
\includegraphics[width=0.7\textwidth]{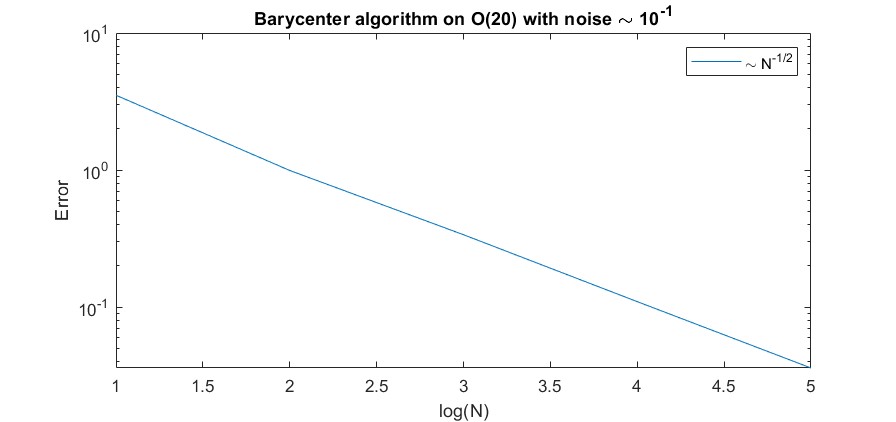}
\caption{Simulation an AR(1) process using barycenters on $O_{20}$ where the noise is of approximate amplitude $10^{-1}$.}
\label{fig:Liegroupaverage}
\end{figure}

For the case when $M=\St_{n,k}$ Algorithm \ref{alg:stiefelgrad} with objective function given in Section \ref{sec:stief} gives an estimate and the error once again is of the order $\mathcal{O}(\sqrt{N}^{-1})$. This is illustrated in Figure \ref{fig:graddescentst} for the case $\St_{15,5}$. In Figure \ref{fig:graddescentnoise} we can see that for $\St_{20,5}$ the error is dependent on the size of the noise, where $N=60$.

\begin{figure}[H]
\centering
\begin{subfigure}[t]{0.49\textwidth}
\captionsetup{width=.95\linewidth}
\includegraphics[width=0.9\textwidth]{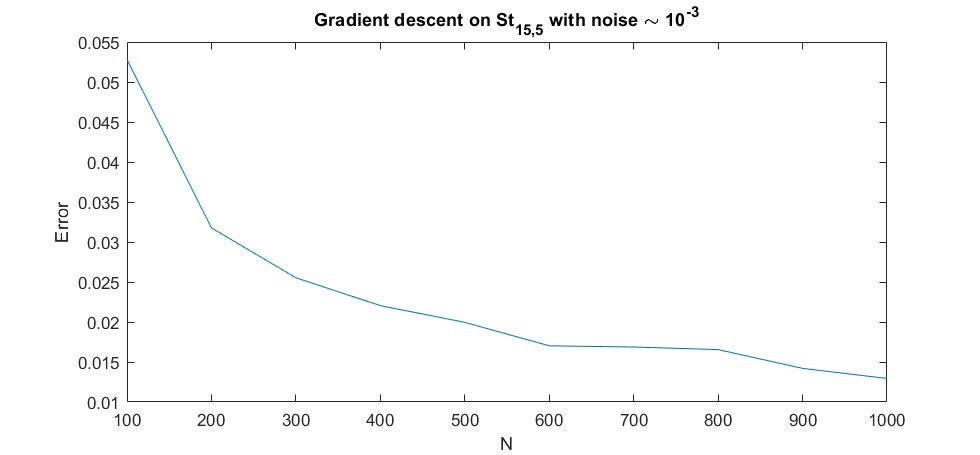}
\caption{Simulation of the gradient descent method on $\St_{15,5}$ where the noise is of approximate amplitude $10^{-3}$.}
\label{fig:graddescentst}    
\end{subfigure}
\begin{subfigure}[t]{0.49\textwidth}
\captionsetup{width=.95\linewidth}
\includegraphics[width=0.9\textwidth]{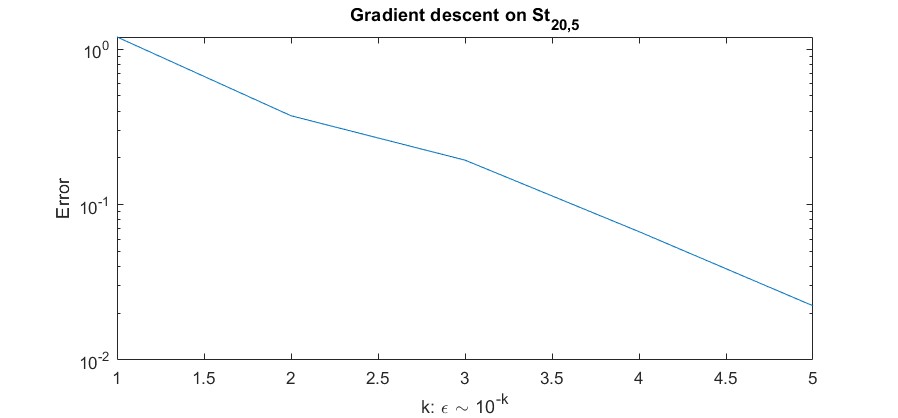}
\caption{Simulation of the gradient descent method on $\St_{20,5}$ over different sizes of the noise.}
\label{fig:graddescentnoise}
\end{subfigure}
\caption{Simulations of the gradient descent on Stiefel manifolds.}
\end{figure}

Another observation is how the error term is dependent on the dimensional parameters $n$ and $k$. In Figure \ref{fig:graddescentstoverd} we fix $k=5$ and $N=200$, and we see an increasing tendency of the error when $n=d$ gets larger.
The opposite behavior can be seen when increasing the parameter $k$ in $\St_{n,k}$ when $k \lesssim n/2$ as seen in Figure \ref{fig:graddescentstoverk} for $\St_{20,k}$ and $N=200$. This is the case even though the dimension of $\St_{n,k}$ increases, the degrees of freedom of the parameter $\Phi$ decreases when $k$ increases.

\begin{figure}[H]
\centering
\begin{subfigure}[t]{0.49\textwidth}
\captionsetup{width=.95\linewidth}
\includegraphics[width=\textwidth]{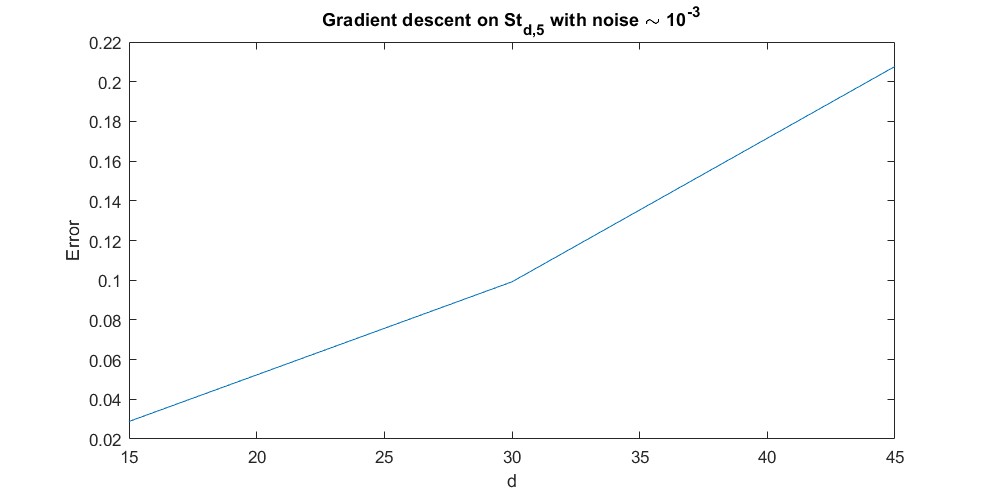}
\caption{ Simulation of the gradient descent method on $\St_{d,5}$ over different $d$ where the noise is of approximate amplitude $10^{-3}$.}
\label{fig:graddescentstoverd}
\end{subfigure}
\begin{subfigure}[t]{0.49\textwidth}
\captionsetup{width=.95\linewidth}
\includegraphics[width=\textwidth]{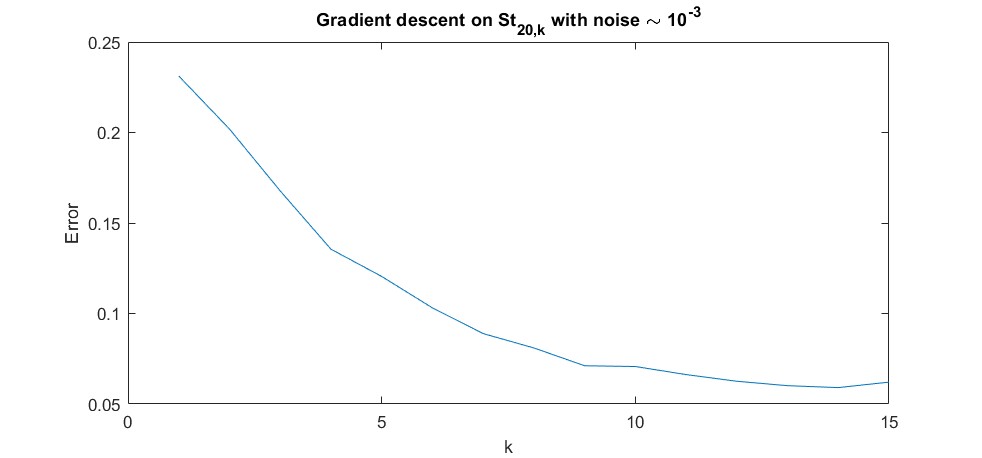}
\caption{Simulation of the gradient descent method on $\St_{20,k}$ over different $k$ where the noise is of approximate amplitude $10^{-3}$.}
\label{fig:graddescentstoverk}    
\end{subfigure}
\caption{Simulations of the gradient descent on various Stiefel manifolds.}
\end{figure}

Lastly, for the case $M=\Gr_{n,k}$, we similarly apply Algorithm \ref{alg:stiefelgrad} with objective function given in Section \ref{sec:grass} to the simulated data set. Once again the error term is of the order $\mathcal{O}(\sqrt{N}^{-1})$, as seen in the case for $\Gr_{15,5}$ and $N=200$ in Figure \ref{fig:graddescentgr}. Furthermore, we have a similar dependency on the noise term $\sigma$, see Figure \ref{fig:graddescentgrnoise} for the case $\Gr_{20,5}$ and $N=60$.

\begin{figure}[H]
\centering
\begin{subfigure}[t]{0.49\textwidth}
\captionsetup{width=.95\linewidth}
\includegraphics[width=\textwidth]{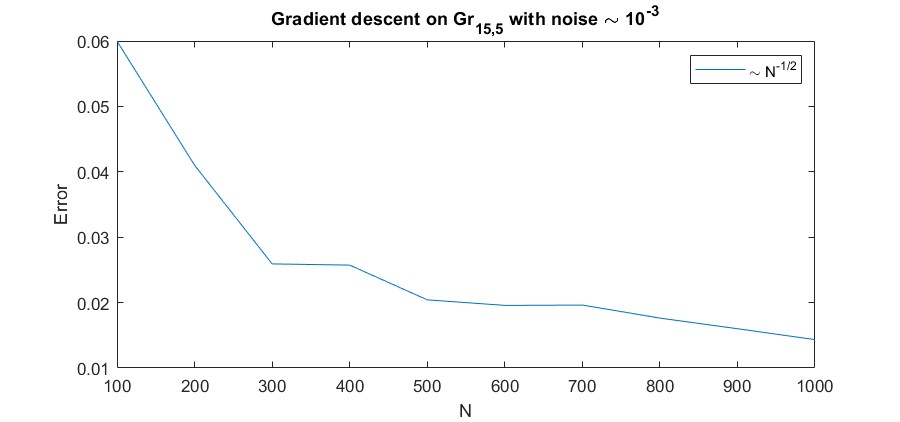}
\caption{Simulation of the gradient descent method on $\Gr_{15,5}$ where the noise is of approximate amplitude $10^{-3}$.}
\label{fig:graddescentgr}    
\end{subfigure}
\begin{subfigure}[t]{0.49\textwidth}
\captionsetup{width=.95\linewidth}
\includegraphics[width=\textwidth]{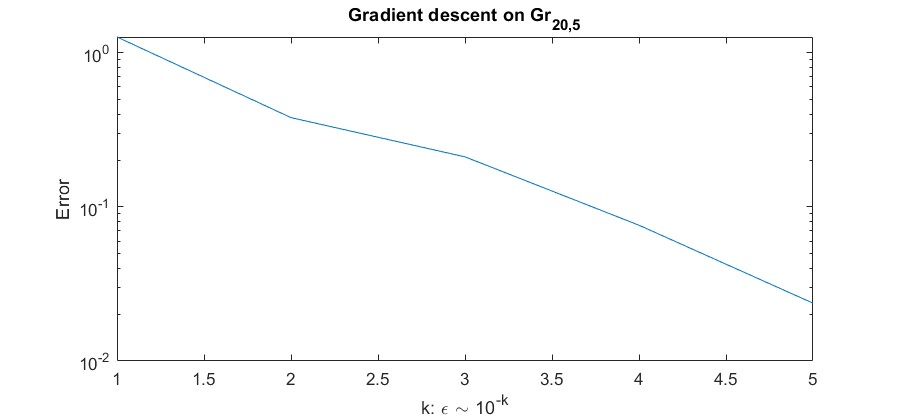}
\caption{Simulation of the gradient descent method on $\Gr_{20,5}$ over different sizes of the noise.}
\label{fig:graddescentgrnoise}    
\end{subfigure}
\caption{Simulations of the gradient descent on Grassmann manifolds.}
\end{figure}

For the same reason as for the Stiefel manifolds, higher degrees of freedom for $\Phi$ means larger error. Fixing $k$ we see the same behavior as previously, see Figure \ref{fig:graddescentgroverd} for the case $k=5$. In Figure \ref{fig:graddescentgroverk} we observe the same behavior as in the case of the Stiefel manifold, for the same reason.

\begin{figure}[H]
\centering
\begin{subfigure}[t]{0.49\textwidth}
\captionsetup{width=.95\linewidth}
    \includegraphics[width=\textwidth]{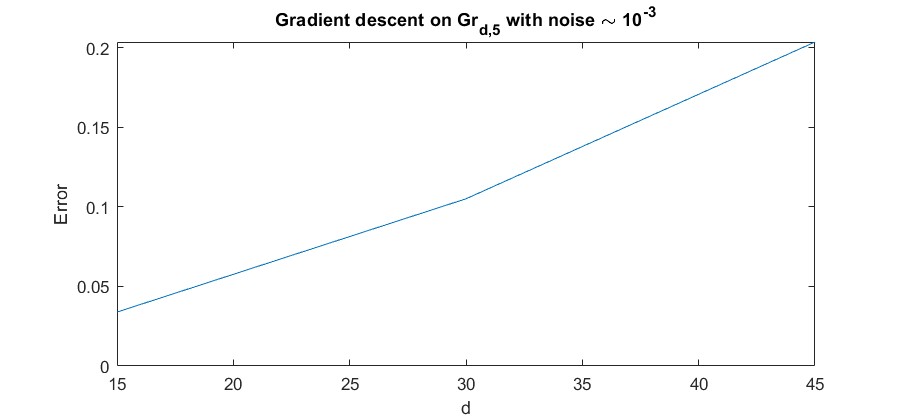}
\caption{Simulation of the gradient descent method on $\Gr_{d,5}$ over different $d$ where the noise is of approximate amplitude $10^{-3}$.}
\label{fig:graddescentgroverd}
\end{subfigure}
\begin{subfigure}
    [t]{0.49\textwidth}
\captionsetup{width=.95\linewidth}
\includegraphics[width=\textwidth]{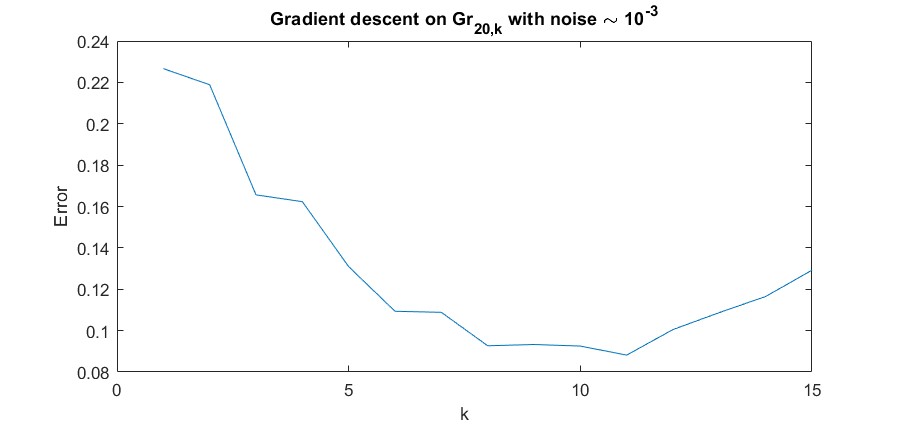}
\caption{Simulation of the gradient descent method on $\Gr_{20,k}$ over different $k$ where the noise is of approximate amplitude $10^{-3}$.}
\label{fig:graddescentgroverk}
\end{subfigure}
\caption{Simulations of the gradient descent on various Grassmann manifolds.}
\end{figure}

\begin{figure}[H]
\centering

\end{figure}

\end{appendices}

\section*{Funding}

This project was partially funded by Huawei technologies.

\section*{Data Availability}
Data sharing not applicable to this article, as no datasets were generated or analyzed during the current study.

\section*{Conflicts of Interest}
The authors declare no conflicts of interest.

\printbibliography

@inproceedings{xavier2006generalization,
  title={On the generalization of AR processes to Riemannian manifolds},
  author={Xavier, Joao and Manton, Jonathan H},
  booktitle={2006 IEEE International Conference on Acoustics Speech and Signal Processing Proceedings},
  volume={5},
  year={2006},
  organization={IEEE},
  pages={1005--1008}
}

@book{lee2009manifolds,
  title={Manifolds and Differential Geometry},
  author={Lee, Jeffrey M},
  volume={107},
  year={2009},
  publisher={American Mathematical Soc.}
}

@inproceedings{solo2014approach,
  title={An approach to stochastic system identification in Riemannian manifolds},
  author={Solo, Victor},
  booktitle={53rd IEEE Conference on Decision and Control},
  pages={6510--6515},
  year={2014},
  organization={IEEE}
}

@article{li2013riemannian,
  title={Riemannian distances for signal classification by power spectral density},
  author={Li, Yili and Wong, Kon Max},
  journal={IEEE Journal of Selected Topics in Signal Processing},
  volume={7},
  number={4},
  pages={655--669},
  year={2013},
  publisher={IEEE}
}

@article{schwarz2014predictive,
  title={Predictive quantization on the Stiefel manifold},
  author={Schwarz, Stefan and Rupp, Markus},
  journal={IEEE Signal Processing Letters},
  volume={22},
  number={2},
  pages={234--238},
  year={2014},
  publisher={IEEE}
}

@incollection{pennec2013exponential,
  title={Exponential Barycenters of the Canonical Cartan Connection and Invariant Means on Lie groups},
  author={Pennec, Xavier and Arsigny, Vincent},
  booktitle={Matrix Information Geometry},
  pages={123--166},
  year={2013},
  publisher={Springer}
}

@article{jauch2020random,
  title={Random orthogonal matrices and the Cayley transform},
  author={Jauch, Michael and Hoff, Peter D and Dunson, David B},
  journal={Bernoulli},
  volume={26},
  number={2},
  pages={1560--1586},
  year={2020},
  publisher={Bernoulli Society for Mathematical Statistics and Probability}
}

@article{zhu2019matrix,
  title={On matrix exponentials and their approximations related to optimization on the Stiefel manifold},
  author={Zhu, Xiaojing and Duan, Chunyan},
  journal={Optimization Letters},
  volume={13},
  number={5},
  pages={1069--1083},
  year={2019},
  publisher={Springer}
}

@article{edelman1998geometry,
  title={The geometry of algorithms with orthogonality constraints},
  author={Edelman, Alan and Arias, Tom{\'a}s A and Smith, Steven T},
  journal={SIAM Journal on Matrix Analysis and Applications},
  volume={20},
  number={2},
  pages={303--353},
  year={1998},
  publisher={SIAM}
}

@article{yuan2019global,
  title={Global optimization with orthogonality constraints via stochastic diffusion on manifold},
  author={Yuan, Honglin and Gu, Xiaoyi and Lai, Rongjie and Wen, Zaiwen},
  journal={Journal of Scientific Computing},
  volume={80},
  number={2},
  pages={1139--1170},
  year={2019},
  publisher={Springer}
}

@article{pennec2006intrinsic,
  title={Intrinsic statistics on Riemannian manifolds: Basic tools for geometric measurements},
  author={Pennec, Xavier},
  journal={Journal of Mathematical Imaging and Vision},
  volume={25},
  number={1},
  pages={127--154},
  year={2006},
  publisher={Springer}
}

@INPROCEEDINGS{Kim2020stochcons,  
author={Kim, Jeongho and Kang, Myeongju and Kim, Dohyun and Ha, Seung-Yeal and Yang, Insoon},  
booktitle={2020 59th IEEE Conference on Decision and Control (CDC)},   
title={A Stochastic Consensus Method for Nonconvex Optimization on the Stiefel Manifold},   
year={2020}, 
pages={1050-1057}
}

@article{zimmermann2017matrix,
  title={A matrix-algebraic algorithm for the Riemannian logarithm on the Stiefel manifold under the canonical metric},
  author={Zimmermann, Ralf},
  journal={SIAM Journal on Matrix Analysis and Applications},
  volume={38},
  number={2},
  pages={322--342},
  year={2017},
  publisher={SIAM}
}

@article{xue2020blind,
  title={Blind Data Detection in Massive MIMO via {$\ell_3$}-Norm Maximization Over the Stiefel Manifold},
  author={Xue, Ye and Shen, Yifei and Lau, Vincent KN and Zhang, Jun and Letaief, Khaled B},
  journal={IEEE Transactions on Wireless Communications},
  volume={20},
  number={2},
  pages={1411--1424},
  year={2020},
  publisher={IEEE}
}

@article{fiori2009algorithm,
  title={An algorithm to compute averages on matrix Lie groups},
  author={Fiori, Simone and Tanaka, Toshihisa},
  journal={IEEE Transactions on Signal Processing},
  volume={57},
  number={12},
  pages={4734--4743},
  year={2009},
  publisher={IEEE}
}

@article{dubey2019frechet,
  title={Fr{\'e}chet analysis of variance for random objects},
  author={Dubey, Paromita and M{\"u}ller, Hans-Georg},
  journal={Biometrika},
  volume={106},
  number={4},
  pages={803--821},
  year={2019},
  publisher={Oxford University Press}
}

@article{kong2022momentum,
  title={Momentum Stiefel Optimizer, with Applications to Suitably-Orthogonal Attention, and Optimal Transport},
  author={Kong, Lingkai and Wang, Yuqing and Tao, Molei},
  journal={arXiv preprint arXiv:2205.14173},
  year={2022}
}

@incollection{li2018wavefront,
  title={Wavefront Orientation Estimation Based on Progressive Bingham Filtering},
  author={Li, K. and Frisch, D. and Radtke, S. and Noack, B. and Hanebeck, U. D.},
  booktitle={2018 Sensor Data Fusion: Trends, Solutions, Applications (SDF)},
  pages={1--6},
  year={2018},
  organization={IEEE}
}

@article{Nunez-Gutierrez-Escarela-2011,
  title={A {B}ayesian regression model for circular data based on the projected normal distribution},
  author={Nu{\~n}ez-Antonio, G. and Guti{\'e}rrez-Pe{\~n}a, E. and Escarela, G.},
  journal={Statistical Modeling},
  volume={11},
  pages={185--201},
  year={2011}
}

@article{nunez2015bayesian,
  title={Bayesian nonparametric models of circular variables based on Dirichlet process mixtures of normal distributions},
  author={Nu{\~n}ez-Antonio, G. and Aus{\'\i}n, M. C. and Wiper, M. P.},
  journal={Journal of Agricultural, Biological, and Environmental Statistics},
  volume={20},
  pages={47--64},
  year={2015},
  publisher={Springer}
}

@article{mardia2025fisher,
  title={Fisher’s legacy of directional statistics, and beyond to statistics on manifolds},
  author={Mardia, Kanti V},
  journal={Journal of Multivariate Analysis},
  volume={207},
  pages={105404},
  year={2025},
  publisher={Elsevier}
}

@article{fisher1953dispersion,
  title={Dispersion on a sphere},
  author={Fisher, Ronald Aylmer},
  journal={Proceedings of the Royal Society of London. Series A. Mathematical and Physical Sciences},
  volume={217},
  number={1130},
  pages={295--305},
  year={1953},
  publisher={The Royal Society London}
}

@article{bowers2000directional,
  title={Directional statistics of the wind and waves},
  author={Bowers, John A and Morton, Ian D and Mould, Gillian I},
  journal={Applied Ocean Research},
  volume={22},
  number={1},
  pages={13--30},
  year={2000},
  publisher={Elsevier}
}

@article{karaibrahimoglu2021circular,
  title={Circular analyses of dates on patients with gastric carcinoma},
  author={Karaibrahimoglu, Adnan and Ayhan, Seren and Karaagac, Mustafa and Artac, Mehmet},
  journal={Journal of Applied Statistics},
  volume={48},
  number={13-15},
  pages={2931--2943},
  year={2021},
  publisher={Taylor \& Francis}
}

@article{demir2019application,
  title={Application of circular statistics to life science data},
  author={Demir, Y{\i}ld{\i}r{\i}m and Bilgin, {\"O}mer C},
  journal={Medical Science and Discovery},
  volume={6},
  number={3},
  pages={63--72},
  year={2019},
  publisher={Zafer AKAN}
}

@inproceedings{kanamori2012non,
  title={Non-convex optimization on Stiefel manifold and applications to machine learning},
  author={Kanamori, Takafumi and Takeda, Akiko},
  booktitle={International Conference on Neural Information Processing},
  pages={109--116},
  year={2012},
  organization={Springer}
}

@inproceedings{schwarz2012adaptive,
  title={Adaptive channel direction quantization based on spherical prediction},
  author={Schwarz, Stefan and Rupp, Markus},
  booktitle={2012 IEEE International Conference on Communications (ICC)},
  pages={3757--3762},
  year={2012},
  organization={IEEE}
}

@article{mccormack2023equivariant,
  title={Equivariant estimation of Fr{\'e}chet means},
  author={McCormack, Andrew and Hoff, Peter D},
  journal={Biometrika},
  volume={110},
  number={4},
  pages={1055--1076},
  year={2023},
  publisher={Oxford University Press}
}

@article{csikos2010extendability,
  title={On the extendability of the Kneser--Poulsen conjecture to Riemannian manifolds.},
  author={Csik{\'o}s, Bal{\'a}zs and Kunszenti-Kov{\'a}cs, D{\'a}vid},
  journal={Advances in Geometry},
  volume={10},
  number={2},
  year={2010}
}

@article{csikos2011volume,
  title={On the volume of the intersection of two geodesic balls},
  author={Csik{\'o}s, Bal{\'a}zs and Horv{\'a}th, M{\'a}rton},
  journal={Differential Geometry and its Applications},
  volume={29},
  number={4},
  pages={567--576},
  year={2011},
  publisher={Elsevier}
}

@article{bouchard2025beyond,
  title={Beyond $ R $-barycenters: an effective averaging method on Stiefel and Grassmann manifolds},
  author={Bouchard, Florent and Laurent, Nils and Said, Salem and Le Bihan, Nicolas},
  journal={IEEE Signal Processing Letters},
  year={2025},
  publisher={IEEE}
}

@article{dohira1995geodesics,
  title={Geodesics in reductive homogeneous spaces},
  author={Dohira, Ryoichi},
  journal={Tsukuba Journal of Mathematics},
  volume={19},
  number={1},
  pages={233--243},
  year={1995},
  publisher={JSTOR}
}

@article{d1975geodesic,
  title={Geodesic spheres and symmetries in naturally reductive spaces.},
  author={D'atri, Joseph E},
  journal={Michigan Mathematical Journal},
  volume={22},
  number={1},
  pages={71--76},
  year={1975},
  publisher={University of Michigan, Department of Mathematics}
}

\end{document}